\documentclass[12pt,reqno]{amsart}
\pdfoutput=1
\usepackage[paperheight=11in,paperwidth=8.5in,left=1in,top=1.25in,right=1in,bottom=1.25in,]{geometry}

\usepackage{hyperref}
\hypersetup{
	colorlinks,
	citecolor=black,
	filecolor=black,
	linkcolor=black,
	urlcolor=black
}
\usepackage[capitalize, nameinlink, noabbrev, nosort]{cleveref}
\usepackage{float}

\usepackage[maxbibnames=99,sorting=nyt,giveninits,maxnames=10,backend=bibtex,style=alphabetic,block=space,url=false]{biblatex}
\usepackage{graphicx} 
\usepackage{enumerate}
\usepackage{tikz}
\usetikzlibrary{shapes.geometric}
\usetikzlibrary{math}
\usetikzlibrary{calc}
\usetikzlibrary{arrows}
\usetikzlibrary{decorations.markings}
\usepackage{tikz-3dplot}
\usepackage{caption}
\tikzset{->-/.style={decoration={
  markings,
  mark=at position .5 with {\arrow{>}}},postaction={decorate}}}

\usepackage{mathrsfs,amssymb,xcolor,mathtools}

\newtheorem{thm}{Theorem}[section]
\newtheorem{lemm}[thm]{Lemma}

\newtheorem{prop}[thm]{Proposition}
\newtheorem{coro}[thm]{Corollary}

\theoremstyle{definition}
\newtheorem{defi}[thm]{Definition}
\newtheorem{exem}[thm]{Example}
\newtheorem{rema}[thm]{Remark}

\newcommand{\R}{\mathbb R}

\newcommand{\Z}{\mathbb Z}
\newcommand{\N}{\mathcal{N}}
\renewcommand{\P}{\mathbf P}
\newcommand{\J}{\mathcal J}
\newcommand{\B}{\mathcal B}
\newcommand{\y}{\mathbf y}
\newcommand{\T}{\mathcal T}
\newcommand{\Grk}{\mathrm{Gr}_{k,n}^{\geq0}}
\newcommand{\precdot}{\prec\hspace{-5pt}\cdot ~}
\newcommand{\Le}{\reflectbox{L}}

\def\des{\ensuremath{\mathrm{des}}}
\def\cdes{\ensuremath{\mathrm{cdes}}}
\def\cDes{\ensuremath{\mathrm{cDes}}}

\def\area{\ensuremath{\mathrm{area}}}
\def\cover{\ensuremath{\mathrm{cover}}}

\DeclareMathOperator{\Ext}{Ext}

\DeclareMathOperator{\conv}{convex}

\DeclareMathOperator{\col}{col}

\DeclarePairedDelimiter{\floor}{\lfloor}{\rfloor}

\newcommand \acknowledgements{\paragraph{\textbf{Acknowledgements}}}

\newcommand{\pyramid}{
    \tikz[scale=0.3]{
    \coordinate (A) at (0,1);
    \coordinate (B) at (2.5,0);
    \coordinate (C) at (4,1);
    \coordinate (D) at (1.5,2);

    \coordinate (Apex) at (2,4);

    \draw[thick] (A) -- (B) -- (C);
    \draw[thick,dotted] (C) -- (D) --(A);

    \draw[thick] (Apex) -- (A);
    \draw[thick] (Apex) -- (B);
    \draw[thick] (Apex) -- (C);
    \draw[thick,dotted] (Apex) -- (D);
    }
}

\newcommand{\prism}{
\tikz[baseline=(base),scale=1.5]{
        \coordinate (A1) at (0,0);        
        \coordinate (B1) at (0.5,0);      
        \coordinate (C1) at (0.25,0.4);   
        
        \coordinate (A2) at (0.2,0.2);    
        \coordinate (B2) at (0.7,0.2);    
        \coordinate (C2) at (0.45,0.6);   

        \draw[thick] (A1) -- (B1) -- (C1) -- cycle;

        \draw[thick,dotted] (A1) -- (A2);
        \draw[thick,dotted] (A2) -- (B2);
        \draw[thick,dotted] (A2) -- (C2);
        \draw[thick] (B2) -- (C2);

        \draw[thick] (B1) -- (B2);
        \draw[thick] (C1) -- (C2);

        \path (A1) -- (B1) coordinate [pos=0.5] (base);
    }
}

\title{The Ehrhart $h^*$-polynomials of positroid polytopes}
\author{Yuhan Jiang}
\email{yjiang@math.harvard.edu}

\begin{document}

\begin{abstract}
A positroid is a matroid realized by a matrix such that all maximal minors are non-negative.
Positroid polytopes are matroid polytopes of positroids.
In particular, they are lattice polytopes.
The Ehrhart polynomial of a lattice polytope counts the number of integer points in the dilation of that polytope.
The Ehrhart series is the generating function of the Ehrhart polynomial, which is a rational function with the numerator called the $h^*$-polynomial.
We compute the $h^*$-polynomials of an arbitrary positroid polytope by a family of shelling orders of it.
We also compute the $h^*$-polynomial of any positroid polytope with some facets removed and we relate it to the descents of permutations.
Our result generalizes that of Early, Kim, and Li for hypersimplices.
\end{abstract}

\maketitle
\setcounter{tocdepth}{1}
\tableofcontents

\section{Introduction}\label{intro}

A positroid is a matroid on an ordered set realized by a matrix such that all of its maximal minors are non-negative. 
Postnikov \cite{postnikov2006} showed that positroids are in bijection with several interesting classes of combinatorial objects, including Grassmann necklaces, decorated permutations, $\Le$-diagrams, and equivalence classes of plabic graphs.

If $P \subseteq \Z^n$ is a $d$-dimensional lattice polytope, its \emph{Ehrhart function/polynomial} is defined for every integer $t \ge 0$ by
\[
E(P,t) := \#(t \cdot P) \cap \Z^n
\]
where $t \cdot P$ is the dilation of $P$ by a factor $t$, i.e., $t \cdot P = \{t \cdot v \mid v \in P\}$.
It is well known from Ehrhart \cite{ehrhart} that $E(P,t)$ is a polynomial function in $t$.
The corresponding \emph{Ehrhart series} is defined as
$\sum_{t=0}^\infty E(P,t) z^t = \frac{h^*(P,z)}{(1-z)^{d+1}}$
where $h^*(P,z) = h_0+h_1 z +\cdots + h_d z^d$ is a polynomial of degree at most $d$ with non-negative coefficients \cite{stanley80}, called the \emph{Ehrhart $h^*$-polynomial} of $P$.
Ehrhart theory naturally extends to half-open polytopes, which are polytopes with some facets removed. The Ehrhart $h^*$-polynomial of the whole polytope can then be obtained by inclusion-exclusion on the faces; see \cref{prop:from_half-open}.

The hypersimplex $\Delta_{k,n}$ is the matroid polytope of the uniform matroid $U_{k,n}$, which is also a positroid.
The $h^*$-polynomial of hypersimplices was first computed by Katzman in \cite{Katzman}. 
Early conjectured that the $h^*$-vector is given by winding numbers of \emph{hypersimplicial decorated ordered set partitions} in \cite{early2017conjectures} and Kim proved it in \cite{Kimh*}. 
Li computed the $h^*$-polynomial of half open hypersimplices in terms of cover relations in \cite{NanLi}. 

In this paper, we generalize their results and give the $h^*$-polynomials of positroid polytopes and half-open positroid polytopes.
Apart from the special case of the uniform matroid, before this paper, no combinatorial formula for the $h^*$-polynomial of an arbitrary positroid polytope was known.

\begin{thm}\label{thm:main}
Let $P_\J$ be any connected positroid polytope (see \cref{def:connected}), where $\J$ is the associated Grassmann necklace (see \cref{def:Grassneck}). 
Let $D_\J \subset S_n$ be the subset of permutations that label the circuit triangulation of $P_\J$ (see \cref{thm:trian}).
For any $w_0 \in D_\J$, let $(\mathcal{P}_{w_0,\J}, \prec)$ be the corresponding poset on $D_\J$ (see \cref{def:digraph}).
The cover statistic of $\mathcal{P}_{w_0,\J}$ gives the $h^*$-polynomial of $P_\J$, i.e.,
$$h^*(P_\J,z) = \sum_{w \in D_\J} z^{\cover(w)}$$
where
$\cover(w)$ is the number of elements $w$ covers in the poset $\mathcal{P}_{w_0,\J}$.
\end{thm}

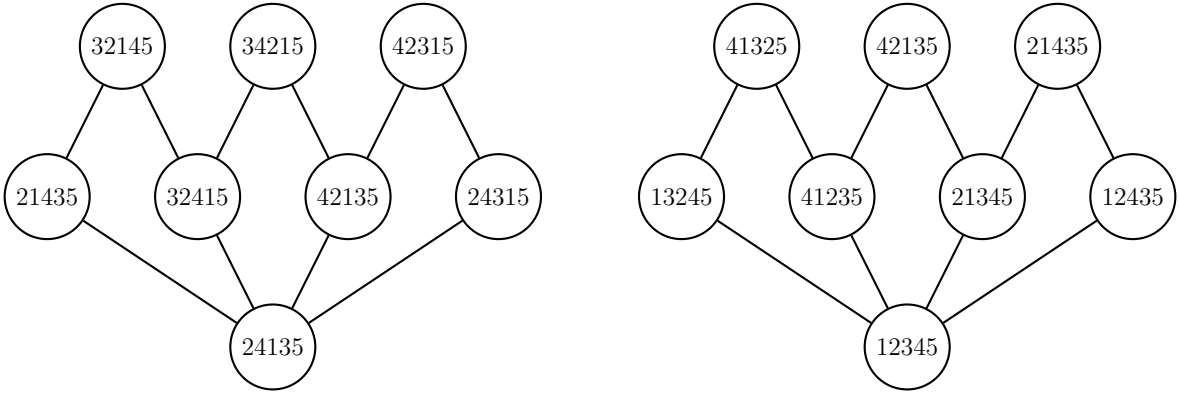
\begin{figure}[H]
    \centering
    \begin{tikzpicture}[thick,scale=1, every node/.style={scale=.8}]
    \draw (0,0) node (2413) [draw,circle] {24135};
    \draw (-3,2) node (2143) [draw,circle] {21435};
    \draw (-1,2) node (3241) [draw,circle] {32415};
    \draw (1,2) node (4213) [draw,circle] {42135};
    \draw (3,2) node (2431) [draw,circle] {24315};
    \draw (2413) -- (2143);
    \draw (2413) -- (3241);
    \draw (2413) -- (4213);
    \draw (2413) -- (2431);
    \draw (-2,4) node (3214) [draw,circle] {32145};
    \draw (0,4) node (3421) [draw,circle] {34215};
    \draw (2,4) node (4231) [draw,circle] {42315};
    \draw (2143) -- (3214) -- (3241);
    \draw (3241) -- (3421) -- (4213);
    \draw (4213) -- (4231) -- (2431);
    \end{tikzpicture}
    \caption{We show the graph of the circuit triangulation of the positroid polytope $P_\J$ associated to the positroid with Grassmann necklace $\J = (123,235,345,145,125)$, which coincides with the Hasse diagram of the poset $\mathcal{P}_{24135,\J}$. The $h^*$-polynomial of $P_\J$ is $1+4z+3z^2$.}
    \label{fig:positroid34521}
\end{figure}

\subsection{Organization}
In \cref{sec:posi}, we introduce positroids and related combinatorial objects. We reduce the problem to \emph{connected positroids}; see \cref{def:connected}.
In \cref{sec:trian}, we analyze the \emph{circuit triangulation} of connected positroid polytopes.
In \cref{sec:shelling}, we give a family of shellings of connected positroid polytopes, which give formulas for the $h^*$-polynomial of all connected positroid polytopes, proving \cref{thm:main}.
In \cref{sec:half-open}, we give the $h^*$-polynomial of all half-open connected positroid polytopes in terms of permutation descents; see \cref{thm:half-open}.
In \cref{sec:tree}, we define a \emph{tree positroid} to be a positroid whose plabic graph is a tree (acyclic), and prove a corollary of \cref{thm:half-open} in the case of tree positroids; see \cref{coro:main}.
\vspace{.5cm}
\acknowledgements{We thank Lauren Williams for suggesting the problem and helpful conversations and comments on the manuscript. We thank Alex Postnikov, Alejandro Morales, Arvind Ayyer, Elisabeth Bullock, Colin Defant, Theo Douvropoulos, Nick Early, and Johnny Gao for helpful discussions.}

\section{Positroids}\label{sec:posi}
A positroid is a matroid that can be represented by a matrix with nonnegative maximal minors.
In this section, we will start by defining \emph{matroids} and \emph{positroids}.
We will define Postnikov's notion of
\emph{Grassmann necklace}, and explain how each one naturally labels a positroid.
Then we introduce positroid polytopes and the notion of \emph{connected} positroids.

\begin{defi}
    A \emph{matroid} is a pair $M = (E,\B)$ consisting of a finite set $E$ and a nonempty collection of subsets $\B = \B(M)$ of $E$, called the \emph{bases} of $M$,
    which satisfy the \emph{basis exchange axiom}:
\begin{center}
    For any $I, J \in \B$ and $i \in I$ there exists $j \in J$ such that $(I \setminus \{i\}) \cup \{j\} \in M$.
\end{center}
All the bases $B \in \B$ have the same size, which is called the \emph{rank} of $M$.
\end{defi}

For a $k \times n$-matrix $A$ of rank $k$ and a $k$-element subset $I \subset [n]$, let $A_I$ denote the $k\times k$-submatrix of $A$ in the column set $I$, and let $\Delta_I(A) := \det(A_I)$ denote the corresponding \emph{maximal minor} of $A$.
The set of $k$-subsets $I \subset [n]$ such that $\Delta_I(A) \neq 0$ form the bases of a rank $k$ matroid $M(A)$.

\begin{defi}
Suppose $A$ is a $k \times n$ matrix of rank $k$ with real entries such that all its maximal minors are nonnegative.
Then the matroid $M(A)$ associated to $A$ is called a \emph{positroid}.
\end{defi}

\subsection{Grassmann necklaces}

\begin{defi}\label{def:Grassneck}
Let $k \leq n$ be a positive integer. A \emph{Grassmann necklace} of type $(k,n)$ is a sequence $(J_1, J_2, \dots, J_n)$ of $k$-subsets $J_i \in \binom{[n]}{k}$ such that for any $i \in [n]$
\begin{itemize}
    \item if $i \in J_i$ then $J_{i+1} = J_i - \{i\} \cup \{j\}$ for some $j \in [n]$,
    \item if $i \notin J_i$ then $J_{i+1} = J_i$,
\end{itemize}
where the indices $i$ are taken modulo $n$.
\end{defi}

The \emph{i-order} $<_i$ on the set $[n]$ is the total order
\[
i <_i i+1 <_i \cdots <_i n <_i 1 <_i \cdots <_i i-2 <_i i-1.
\]

Let $i \in [n]$. The \emph{Gale order} on $\binom{[n]}{d}$ (with respect to $<_i$) is the partial order $\le_i$ defined as follows: for any two $d$-subsets $S = \{s_1 <_i \cdots <_i s_d \} \subseteq [n]$ and $T = \{t_1 <_i \cdots <_i t_d \} \subseteq [n]$, we have $S \le_i T$ if and only if $s_j \le_i t_j$ for all $j \in [d]$.

For a matroid $M \subseteq \binom{[n]}{k}$ of rank $k$ on the set $[n]$, let $\J_M = (J_1, \dots, J_n)$ be the sequence of subsets in $[n]$ such that, for $i \in [n]$, $J_i$ is the lexicographically minimal basis of $M$ with respect to the Gale order with respect to $<_i$ on $[n]$.

\begin{lemm}[\cite{postnikov2006}, Lem 16.3]
For any matroid of rank $k$, the sequence $\J(M)$ is a Grassmann necklace of type $(k,n)$.
\end{lemm}

\begin{thm}[\cite{postnikov2006,oh2017}]\label{thm:basis}
Let $\J = (J_1, \dots, J_n)$ be a Grassmann necklace of type $(k,n)$. Then the collection
$$ \mathcal{B}(\J) := \left\{ B \in \binom{[n]}{k} \mid B \geq_i J_i \text{ for all } i \in [n] \right\} $$
is the collection of bases of a rank $k$ positroid $\mathcal{M}(\J) := ([n], \mathcal{B}(\J))$.
Moreover, for any positroid $M$ we have $\mathcal{M}(\J(M)) = M$.
\end{thm}

\begin{exem}
Let $\J$ be the Grassmann necklace $(12,23,13,14)$. The bases of the positroid associated to $\J$ is $\{12,13,14,23,24\}$.
\end{exem}

\subsection{Positroid polytopes}

\begin{defi}
Given a matroid $M = ([n],\B)$, the (basis) \emph{matroid polytope} $P_M$ of $M$ is the convex hull of the indicator vectors of the bases of $M$:
$$P_M := \conv \{e_B \mid B \in \B\} \subset \R^n$$
where $e_B = \sum_{i \in B} e_i$ and $\{e_1,\dots,e_n\}$ is the standard basis of $\R^n$.
\end{defi}

The next proposition provides inequalities that define the matroid polytope of a positroid.
\begin{prop}[\cite{positroids,polyposi}]\label{prop:posi}
Let $\J = (J_1, J_2, \dots, J_n)$ be a Grassmann necklace of type $(k,n)$. 
For any $i \in [n]$, suppose the elements of $J_i$ are $a_1^i <_i a_2^i <_i \cdots <_i a_k^i$. Then the matroid polytope $P_\J$ of the positroid associated to $\J$ can be described by the inequalities
\begin{align*}
    x_1+x_2+\cdots+x_n &= k & \\
    x_i &\ge 0 & \text{ for all }i \in [n], \\
    x_i + x_{i+1} + \cdots + x_{a_{j}^{i}-1} &\le j-1 & \text{ for all } i \in [n] \text{ and } j \in [k],
\end{align*}
where all the subindices are taken modulo $n$.
\end{prop}

To write the inequalities more concisely, we will use the following notation. Given $i, j \in [n]$, we define the \emph{(cyclic) interval} $[i,j]$ to be the totally ordered set
$$
[i,j] := \left\{ \begin{matrix}
        \{i <_i i+1 <_i \cdots <_i j\} & \text{ if } i \leq j, \\
        \{i <_i i+1 <_i \cdots <_i n <_i 1 <_i \cdots <_i j\} & \text{ if } i > j,
    \end{matrix}
    \right.
$$
and
$$ x_{[i,j]} = x_i + \cdots + x_{j-1}$$
with all indices modulo $n$.

\subsection{Connected positroids}

\begin{defi}\label{def:connected}
A matroid which cannot be written as the direct sum of two nonempty matroids is called \emph{connected}.
\end{defi}

\begin{lemm}[\cite{positroids}]
Let $M$ be a positroid on $[n]$ and write it as a direct sum of connected matroids $M = M_1 \oplus \cdots \oplus M_m$. Then each $M_i$ is a positroid.
\end{lemm}

\begin{rema}\label{rema:reduce_to_connected}
Connected positroids can be characterized in terms of their corresponding \emph{decorated permutations}. Ardila--Rincon--Williams showed that a positroid is connected if and only if its associated decorated permutation is a \emph{stabilized-interval-free} permutation \cite[Cor. 7.9]{positroids}, and they showed that the matroid polytope of a connected positroid on $[n]$ has dimension $n-1$ \cite[Thm. 8.2]{positroids}.

If a matroid $M$ is equal to the direct sum of matroids $M = M_1 \oplus \cdots \oplus M_m$, then the matroid polytope $P_M$ of $M$ is equal to the direct product of the matroid polytopes $P_{M_1} \times \cdots \times P_{M_m}$.

The Ehrhart polynomial of $P_M$ is equal to the product of the Ehrhart polynomials $E(P_M) = E(P_{M_1}) \cdots E(P_{M_m})$.
Thus, to know the $h^*$-polynomial of all positroid polytopes, it suffices to give formulas for all connected positroid polytopes.
From now on, we will focus on connected positroids.
\end{rema}

\section{Circuit triangulation of connected positroid polytopes}
\label{sec:trian}

In this section, we analyze the triangulation of connected positroid polytopes in terms of \emph{$(w)$-simplices}, defined by \cite{parisi2024magic}. Other equivalent characterizations also appear in \cite{alcove1}. We follow the conventions of \cite{parisi2024magic}.

\begin{defi}\label{def:leftcycdes}
Let $w \in S_n$. A letter $i < n$ is a \emph{left descent} of $w$ if $i$ occurs to the right of $i+1$ in $w$. In other words, $w^{-1}(i) > w^{-1}(i+1)$. We say that $i \in [n]$ is a \emph{cyclic left descent} of $w$ if either $i<n$ is a left descent of $w$ or if $i=n$ and 1 occurs to the left of $n$ in $w$, that is, $w^{-1}(1) < w^{-1}(n)$. We let $\cDes_L(w)$ denote the set of cyclic left descents of $w$ and let $\cdes_L(w) = |\cDes_L(w)|$.
\end{defi}

\begin{defi}\label{Dk+1n}
Choose $0 \leq k \leq n-2$. We define $D_{k+1, n}$ to be the set of permutations $w \in S_n$ with $k+1$ cyclic left descents and $w_n=n$. We let $D_n$ be the set of permutations $w \in S_n$ with $w_n=n$. For $w = w_1 \cdots w_n$, let $(w)$ denote the cycle $(w_1, \cdots, w_n)$.
\end{defi}

\begin{rema}\label{rem:cdes}
The definition of cyclic left descent only depends on the total order on $[n]$. That is, given any permutation of any totally ordered set that is not a singleton or empty set, the cyclic left descent of such a permutation can be defined analogously.
This definition coincides with \cite[Def. 6.2]{alcove2} in type A.
\end{rema}

\begin{defi}
Let $w = w_1\cdots w_n \in S_n$ and $i, j \in [n]$.
Let $[i,j]$ denote the cyclic interval defined in \cref{sec:posi}.
Let $w|_{[i,j]}$ be the restriction of $w$ to the totally ordered set $[i,j]$, and let $\cdes_L(w|_{[i,j]})$ be the number of cyclic left descents of $w|_{[i,j]}$, which is well defined by \cref{rem:cdes}.
\end{defi}

\begin{exem}
Let $w = 32415$. Then $w|_{[1,3]} = 321$ and $w|_{[3,1]} = 3415$. Then $\cDes_L(w|_{[1,3]}) = \{1,2\}$ and $\cDes_L(w|_{[3,1]}) = \{5,1\}$
\end{exem}

\begin{defi}\label{def:w-simplex}
For $w = w_1 w_2 \cdots w_n \in S_n$, let $w^{(a)}$ denote the cyclic rotation of $w$ ending at $a$. We define
$$ I_r(w) := \cDes_L(w^{(r)}). $$
Note that $I_r$ only depends on the cycle $(w)$.

We define the $(w)$-simplex
$\Delta_{(w)}$ to be the convex hull of the points $e_{I_1}, \dots, e_{I_n}$; this is an $(n-1)$-dimensional simplex. We call
\[
I_{w_1} \to I_{w_2} \to \cdots \to I_{w_n} \to I_{w_1}
\]
the \emph{circuit} of $\Delta_{(w)}$.
Triangulations by $(w)$-simplices are often called \emph{circuit triangulations}.
\end{defi}

\begin{exem}
The circuit of $32415$ is $135 \to 235 \to 245 \to 124 \to 125 \to 135$. The vertices of $\Delta_{(32415)}$ are $11001, 10101, 01101, 01011, 11010$.
\end{exem}

For a permutation $w = w_1\dots w_n \in S_n$, Stanley \cite{hypercube} defined the simplex
\begin{equation}\label{eq:stanleytriang}
    \nabla_w := \left\{ \y \in [0,1]^n \mid 0 \le y_{w_1} \le \cdots \le y_{w_n} \le 1 \right\}.
\end{equation}
Stanley showed that the $\nabla_w$'s triangulate the hypercube, i.e., $[0,1]^n = \bigcup_{w \in S_n} \nabla_w$

From here on, we essentially follow the conventions of \cite{hypercube,multihyper,AJR20}, but our conventions differ slightly from theirs.

Define a measure preserving map $\phi: [0,1]^n \to [0,1]^n, (y_1,\dots,y_n) \mapsto (x_1,\dots,x_n)$ such that $x_n = 1-y_n$ and
\begin{equation*}
    x_{i+1} = \left\{
    \begin{matrix}
        y_i-y_{i-1} & \text{ if } y_i > y_{i-1} \\
        y_i-y_{i-1} + 1 & \text{ if } y_i < y_{i-1}
    \end{matrix}\right.
\end{equation*}
with inverse
\begin{equation*}
    y_i = 1 + \floor{x_i+\dots+x_n} - (x_i+\dots+x_n).
\end{equation*}

\begin{rema}
In \cite{hypercube}, the map $\phi$ is defined such that $y_1 = x_1$ and $y_i = x_1 + \cdots + x_i - \floor{x_1 + \cdots + x_i}$.
\end{rema}

\begin{lemm}\label{lem:wsimplex}
Let $w \in D_n$. Let $p : \R^n \to \R^{n-1}$ be the projection onto the first $(n-1)$ coordinates.
Then the projected $(w)$-simplex $p(\Delta_{(w)})$ has facet-defining inequalities of the form
\begin{align*}
x_{[w_i, w_{i+1}]} &\ge \cdes_L(w|_{[w_i, w_{i+1}]})-1 & \text{ for } w_i < w_{i+1} \\
x_{[w_{i+1}, w_i]} &\le \cdes_L(w|_{[w_{i+1}, w_i]}) & \text{ for } w_i > w_{i+1},
\end{align*}
for all $i \in [n-2]$ and
\begin{align*}
    x_{[w_1,1]} &\le \cdes_L(w|_{[w_1,n]}) \\
    x_{[w_n,1]} &\ge \cdes_L(w|_{[w_n,n]})-1.
\end{align*}

Moreover, for all $i,j \in [n]$, the projected $(w)$-simplex satisfies
$\cdes_L(w|_{[i,j]})-1 \leq x_{[i,j]} \leq \cdes_L(w|_{[i,j]})$.
\end{lemm}

\begin{proof}
It follows from \cite[Theorem 2.7]{alcove1} that $\phi(\nabla_{\underline{w}}) = p(\Delta_{(w)})$.
The facets of $p(\Delta_{(w)})$ are obtained by applying the transformation $\phi$ to the facets of $\nabla_{\underline{w}}$.
This gives the first statement.

The second statement is a type A special case of \cite[Lem. 8.1]{alcove2}.
\end{proof}

\begin{exem}
The projected $32415$-simplex has facet-defining inequalities
\begin{align*}
    x_1+x_2+x_3+x_4 & \ge 2 \\
    x_3 + x_4 & \le 1 \\
    x_2 & \le 1 \\
    x_2+x_3 & \ge 1 \\
    x_1+x_2+x_3 & \le 2
\end{align*}
\end{exem}

\begin{thm}\label{thm:trian}
Let $P_\J$ be any connected positroid polytope, where $\J = (J_1, \dots, J_n)$ is the associated Grassmann necklace. 

For any $i \in [n]$, suppose the elements of $J_i$ are $a_1^i <_i \cdots <_i a_k^i$. Then the positroid polytope $P_\J$ is triangulated by $(w)$-simplices for $w$ in the set
\begin{align*}
    D_\J &:= \{ w \in D_{k+1,n} \mid \cdes_L(w|_{[i,a_j^i]}) \leq j-1 \text{ for all } i \in [n], j \in [k] \} \\
    &= \{ w \in D_{k+1,n} \mid I_{w_i} \geq_j J_j \text{ for all } i,j \in [n] \},
\end{align*}
where $I_{w_1} \to I_{w_2} \to \dots \to I_{w_n} \to I_{w_1}$ is the circuit of $w$.
\end{thm}

\begin{proof}
By \cite[Theorem 8.2]{positroids}, the positroid polytope $P$ associated to $\pi$ has dimension $n-1$.

By \cref{prop:posi} and \cref{lem:wsimplex}, we have that $P$ is triangulated by $(w)$-simplices for $w$ in
$$D_\J = \{ w \in D_{k+1,n} \mid \cdes_L(w|_{[i,a_j^i]}) \leq j-1 \text{ for all } i \in [n], j \in [k] \}.$$
By \cref{thm:basis}, we obtain the equivalent characterization
\[
D_\J = \{ w \in D_{k+1,n} \mid I_{w_i} \geq_j J_j \text{ for all } i,j \in [n] \}. \tag*{\qedhere}
\]
\end{proof}

\begin{figure}
    \centering
    \begin{tikzpicture}
    \coordinate (A) at (0,1);
    \coordinate (B) at (2.5,0);
    \coordinate (C) at (4,1);
    \coordinate (D) at (1.5,2);

    \coordinate (Apex) at (2,5);

    \draw[thick] (A) -- (B) -- (C);
    \draw[thick,dotted] (C) -- (D) --(A);

    \draw[thick] (Apex) -- (A);
    \draw[thick] (Apex) -- (B);
    \draw[thick] (Apex) -- (C);
    \draw[thick,dotted] (Apex) -- (D);
    \draw[thick,dotted] (B) -- (D);

    \node[below left] at (A) {0110};
    \node[below right] at (B) {1010};
    \node[above right] at (C) {1001};
    \node[above left] at (D) {0101};

    \node[above] at (Apex) {1100};
    \end{tikzpicture}
    \caption{The positroid polytope $P_\J$ associated to the Grassmann necklace $(12,23,13,14)$, with bases $\{12,13,14,23,24\}$.}
    \label{fig:pyramid}
\end{figure}
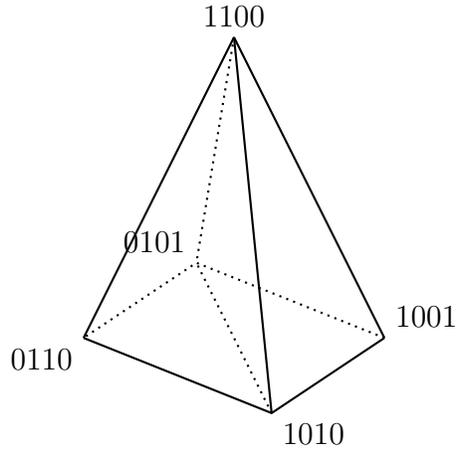

\begin{exem}
Consider the Grassmann necklace $\J = (12,23,13,14)$. Then $D_\J$ consists of permutations in $S_4$ that end with 4 such that $\cdes_L(w|_{[3,4]}) \leq 1$, so $D_\J = \{1324, 2134\}$. The positroid polytope $P_\J$ is a pyramid, as in \cref{fig:pyramid}. The $(1324)$-simplex has vertices $1100,0101,0110,1010$, and the $(2134)$-simplex has vertices $1100,0101,1001,1010$.
\end{exem}

\section{The $h^*$-polynomial of connected positroid polytopes}\label{sec:shelling}

In this section, we prove \cref{thm:main}, which gives the $h^*$-polynomial of an arbitrary connected positroid polytopes.

\begin{defi}
A \emph{shelling} of a triangulation $\Gamma$ is a linear order on its maximal faces $G_1, G_2, \dots, G_s$ such that, for each $i \in [2,s]$, the set $G_i \cap (G_1 \cup \cdots \cup G_{i-1})$ is a union of facets of $G_i$.
\end{defi}

\begin{lemm}[\cite{stanley80}]\label{lem:stanley}
Let $\Gamma$ be a unimodular triangulation of a polytope $P$ and
let $G_1, \dots, G_s$ be a shelling of $\Gamma$. Then the $h^*$-polynomial of $P$ is equal to
$h^*(P,z) = \sum_{i=1}^{s} z^{\alpha_i}$ where
$\alpha_i$ is the number of facets of $G_i$ in the intersection $G_i \cap (G_1 \cup \cdots \cup G_{i-1})$.
\end{lemm}

\begin{defi}\label{def:graph}
Consider a connected positroid polytope $P_\J$ with Grassmann necklace $\J$.
Let $\Gamma_\J$ be the graph whose vertices are $w \in D_\J$ and there is an edge between $w$ and $u$ if and only if $\Delta_{(w)}$ and $\Delta_{(u)}$ share a common facet.
We call $\Gamma_\J$ the \emph{graph of the circuit triangulation of $P_\J$}.
\end{defi}

\begin{rema}
In the special case of the hypersimplex $\Delta_{k,n}$, the graph $\Gamma_\J$ for 
$$\J = ([1,k],[2,k+1],\dots,[n,k-1])$$
coincides with $\Gamma_{k,n}$ in \cite{alcove1}.
For a generic connected positroid with Grassmann necklace $\J$ of type $(k,n)$, the graph $\Gamma_\J$ is a connected subgraph of $\Gamma_{k,n}$.
\end{rema}

\begin{defi}\label{def:digraph}
Let $\Gamma = (V, E)$ be an undirected graph, and let $v_0 \in V$ be an arbitrary vertex of $\Gamma$. 
Define a partial order $(\mathcal{P}_{v_0,\Gamma},\prec)$ on $V$ with minimal element $v_0$ such that, for two distinct vertices $u, v \in V$, $u \prec v$ if and only if there exists a shortest path from $v_0$ to $v$ passing through $u$.
\end{defi}

In particular, the above definition applies to $\Gamma_\J$ and $w_0 \in D_\J$ for any Grassmann necklace $\J$. 
In this case, we will simplify our notation and denote $\mathcal{P}_{w_0, \Gamma_\J}$ by $\mathcal{P}_{w_0,\J}$.
The following theorem characterizes the edges of $\Gamma_\J$.

\begin{thm}[Theorem 2.9, \cite{alcove1}]\label{thm:edge}
Let $\Gamma_{k,n}$ be the graph of the circuit triangulation of the hypersimplex $\Delta_{k,n}$.
Two simplices $\Delta_{(u)}$ and $\Delta_{(w)}$ of $\Gamma_{k,n}$ are adjacent if and only if there exists $i \in [n]$ such that $u_i - u_{i+1} \neq \pm 1 \bmod n$ and the cycle $(w)$ is obtained from $(u)$ by switching $u_i$ with $u_{i+1}$. Here the indices are modulo $n$ so $u_{n+1} = u_1$ by convention.
\end{thm}

\begin{rema}
Let $s_i = (i, i+1)$ for $i \in [n]$ with indices modulo $n$. 
Then $S_n = \langle s_1, \dots, s_n \mid s_i^2, (s_is_j)^2, (s_i s_{i+1})^3, s_1 s_2\cdots s_n \cdots s_2 s_1 \rangle$.
The cycles in $\Gamma_\J$ are generated by the relations among $s_1,\dots,s_n$. \cref{thm:edge} labels the edges of $\Gamma_\J$ by $s_1, \dots, s_n$.
\end{rema}

\begin{defi}[\cite{bjornerbrenti}]
The \emph{affine permutations} of period $n$ are bijections $u$ of $\Z$ such that $u(x+n) = u(x)+n$ for all $x \in \Z$ and $\sum_{x=1}^n u(x) = \binom{n+1}{2}$.
They form a group, denoted by $\tilde{S_n}$, with composition as group operation.
Such a $u$ is uniquely determined by its values on $[n]$, and we write $u = [u(1),\dots, u(n)]$, which we call the \emph{window} notation of $u$.
Let $\tilde{s}_i = [\dots,i+1,i,\dots]$ for $i \in [n-1]$ and $\tilde{s}_n = [0,2,3,\dots,n-1,n+1]$. Then $\tilde{S_n}$ is generated by simple reflections $s_1,\dots,s_n$.

Each affine permutation $w$ can be written as a product of simple reflections $\tilde{s}_1 \cdots \tilde{s}_k$, and if $k$ is minimal among all such expressions for $w$, then $k$ is called the \emph{length} of $w$ (written $\ell(w)=k$), and the word $\tilde{s}_1 \cdots \tilde{s}_k$ is called a \emph{reduced word} for $w$. The \emph{right weak order} on $\tilde{S_n}$ is a partial order such that $u \leq_R w$ if and only if $w = u \tilde{s}_1 \cdots \tilde{s}_k$ such that $\ell(u \tilde{s}_1 \cdots \tilde{s}_i) = \ell(u)+i$ for all $0 \leq i \leq k$.
\end{defi}

\begin{lemm}\label{lem:relabel}
Let $P_\J$ be a connected positroid polytope where $\J$ is the associated Grassmann necklace. Fix $w_0 \in D_\J$. There is a unique injection $L_{w_0}: D_\J \to \tilde{S_n}$ such that $L_{w_0}(w_0) = e$ and $L_{w_0}(w) = \tilde{s}_{i_1} \cdots \tilde{s}_{i_k}$ if there is a path $w_0 \overset{s_{i_1}}{\to} \cdots \overset{s_{i_k}}{\to} w$ in $\Gamma_\J$ for all $w \in D_\J$.
Moreover, for any reduced word $\tilde{s}_{i_1} \cdots \tilde{s}_{i_k}$ of $L_{w_0}(w)$, there is a shortest path $w_0 \overset{s_{i_1}}{\to} \cdots \overset{s_{i_k}}{\to} w$ in $\Gamma_\J$.
\end{lemm}
\begin{proof}
We only need to show that $L_{w_0}$ is well-defined. Suppose there is another path $w_0 \overset{s_{j_1}}{\to} \cdots \overset{s_{j_m}}{\to} w$ from $w_0$ to $w$. Then $w_0 \overset{s_{i_1}}{\to} \cdots \overset{s_{i_k}}{\to} w \overset{s_{j_m}}{\to} \cdots \overset{s_{j_1}}{\to} w_0$ form a cycle in $\Gamma_\J$. 
If the cycle is generated by $(s_is_j)^2$ for $i-j \neq \pm 1 \bmod n$ and $(s_is_{i+1})^3$, then $\tilde{s}_{i_1} \cdots \tilde{s}_{i_k} = \tilde{s}_{j_1} \cdots \tilde{s}_{j_m}$ in $\tilde{S_n}$.

We argue that there is no cycle in $\Gamma_\J$ that corresponds to the relation $s_1 s_2 \cdots s_{n-1} s_n s_{n-1} \cdots s_2 s_1$ among $s_1, \dots, s_n$.
Suppose not, then there is a cycle $(u)$ such that $u_1 - u_i \neq \pm 1 \bmod n$ for all $i \in [n]$, which is impossible.

The second statement follows from the convexity of $P_\J$ by \cite[Prop. 3.5]{alcove2}.
\end{proof}

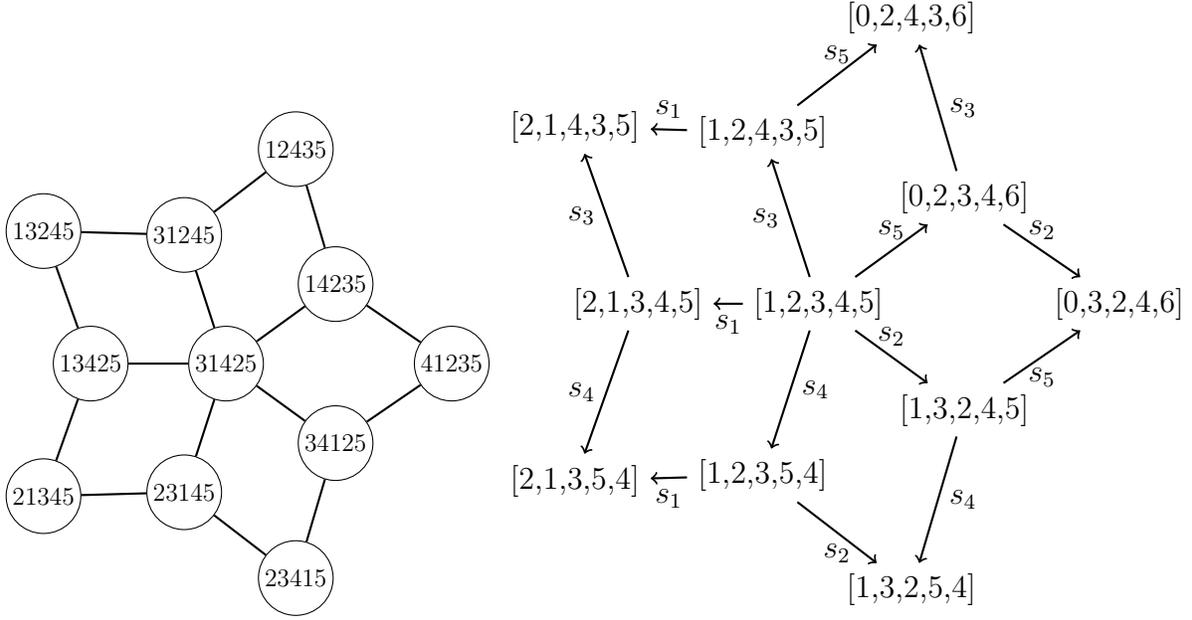
\begin{figure}
    \centering
    \begin{tikzpicture}[scale=3,every node/.style={draw, circle, inner sep=2pt, scale=.8}]
    \node (4123) at (0:1) {41235}; 
    \node (1243) at (72:1) {12435}; 
    \node (1324) at (144:1) {13245};
    \node (2134) at (216:1) {21345};
    \node (2341) at (288:1) {23415}; 
    
    \node (1423) at (36:0.6) {14235}; 
    \node (3124) at (108:0.6) {31245}; 
    \node (1342) at (180:0.6) {13425};
    \node (2314) at (252:0.6) {23145};
    \node (3412) at (324:0.6) {34125}; 

    \node (3142) at (0,0) {31425};

    \draw[thick] (3142) -- (1423);
    \draw[thick] (3142) -- (3124);
    \draw[thick] (3142) -- (1342);
    \draw[thick] (3142) -- (2314);
    \draw[thick] (3142) -- (3412);

    \draw[thick] (3124) -- (1324);
    \draw[thick] (3124) -- (1243);
    
    \draw[thick] (1423) -- (4123);
    \draw[thick] (1423) -- (1243);
    
    \draw[thick] (3412) -- (4123);
    \draw[thick] (3412) -- (2341);
    
    \draw[thick] (2314) -- (2134);
    \draw[thick] (2314) -- (2341);

    \draw[thick] (1342) -- (2134);
    \draw[thick] (1342) -- (1324);
    
\end{tikzpicture}
\begin{tikzpicture}[scale=4]
    \node (4123) at (0:1) {[0,3,2,4,6]}; 
    \node (1243) at (72:1) {[0,2,4,3,6]}; 
    \node (1324) at (144:1) {[2,1,4,3,5]};
    \node (2134) at (216:1) {[2,1,3,5,4]};
    \node (2341) at (288:1) {[1,3,2,5,4]}; 
    
    \node (1423) at (36:0.6) {[0,2,3,4,6]}; 
    \node (3124) at (108:0.6) {[1,2,4,3,5]}; 
    \node (1342) at (180:0.6) {[2,1,3,4,5]};
    \node (2314) at (252:0.6) {[1,2,3,5,4]};
    \node (3412) at (324:0.6) {[1,3,2,4,5]}; 

    \node (3142) at (0,0) {[1,2,3,4,5]};

    \draw[thick,->] (3142) -- (1423) node[midway, above] {$s_5$};
    \draw[thick,->] (3142) -- (3124) node[midway,left] {$s_3$};
    \draw[thick,->] (3142) -- (1342) node[midway,below] {$s_1$};
    \draw[thick,->] (3142) -- (2314) node[midway,right] {$s_4$};
    \draw[thick,->] (3142) -- (3412) node[midway,above] {$s_2$};

    \draw[thick,->] (3124) -- (1324) node[midway,above] {$s_1$};
    \draw[thick,->] (3124) -- (1243) node[midway,above] {$s_5$};
    
    \draw[thick,->] (1423) -- (4123) node[midway,above] {$s_2$};
    \draw[thick,->] (1423) -- (1243) node[midway,right] {$s_3$};
    
    \draw[thick,->] (3412) -- (4123) node[midway,below] {$s_5$};
    \draw[thick,->] (3412) -- (2341) node[midway,right] {$s_4$};
    
    \draw[thick,->] (2314) -- (2134) node[midway,below] {$s_1$};
    \draw[thick,->] (2314) -- (2341) node[midway,below] {$s_2$};

    \draw[thick,->] (1342) -- (2134) node[midway,left] {$s_4$};
    \draw[thick,->] (1342) -- (1324) node[midway,left] {$s_3$};
    
\end{tikzpicture}
 \caption{On the left, we show the graph $\Gamma_{2,5}$ of the circuit triangulation of the hypersimplex $\Delta_{2,5}$. The vertices of $\Gamma_{2,5}$ are labeled by permutations $w \in D_{3,5}$ in one-line notation. On the right, we relabel the vertices $w$ of $\Gamma_{2,5}$ by affine permutations $L_{w_0}(w)$ in window notation with $w_0 = 31425$ according to \cref{lem:relabel}. The arrows represent cover relations in the poset $\mathcal{P}_{31425,\J}$ for $\J = (12,23,34,45,51)$, pointing from a smaller element to a bigger element.}
\label{fig:hypersimplex25}
\end{figure}

\begin{rema}
We can embed any connected positroid polytope into an \emph{affine Coxeter arrangement}, which makes \cref{lem:relabel} obvious; see \cite{humphreys}.
By \cref{lem:relabel}, the partial order $\mathcal{P}_{w_0,\J}$ is identified with the restriction of the weak order on a subset of $\tilde{S_n}$.
\end{rema}

\begin{lemm}\label{lem:facets}
Let $P_\J$ be a connected positroid polytope where $\J$ is the associated Grassmann necklace. 
Let $u, w \in D_\J$. 
Let $J$ be the minimal set of simple reflections that generate $L_u(w)$ in $\tilde{S_n}$ with $|J| = m$. Then the intersection between $\Delta_{(w)}$ and $\Delta_{(u)}$ is $\Delta_{(w)} \cap \Delta_{(u)} = \conv(e_{I_{w_i}} \mid i \notin J)$ with codimension $m$.
\end{lemm}

\begin{proof}
Suppose $I_{u_1} \to I_{u_2} \to \cdots \to I_{u_n} \to I_{u_1}$ is the circuit of $u$.
Then the circuit of $u s_i$ only differs from that of $u$ at $I_{u_i}$.
As the circuit of $u$ gives the vertices of $\Delta_{(u)}$, the statement follows from induction.
\end{proof}

\begin{prop}\label{prop:shelling}
Consider a connected positroid polytope $P_\J$ with Grassmann necklace $\J$. 
Let $\Gamma_\J$ be the graph of the circuit triangulation of $P_\J$.
For any $w_0 \in D_\J$,
any linear extension of $\mathcal{P}_{w_0, \J}$ is a shelling of the circuit triangulation of $P_\J$.
\end{prop}

\begin{rema}
Benedetti--Knauer--Valencia-Porras proved this for general type A alcoved polytopes in \cite[Prop. 2.5]{benedetti2023latticepathmatroidpolytopes} using a more geometric argument.
\end{rema}

\begin{proof}
Let $w_0, w_1, \dots, w_s$ be a linear extension of $(\mathcal{P}_{w_0,\J},\prec)$. 

Fix $i \in [s]$.
Let $S_1$ be the set of $\tilde{s}_k$ for $k \in [n]$ such that $L_{w_0}(w_i) \tilde{s}_k < L_{w_0}(w_i)$ in the weak order of $\tilde{S_n}$ and $(w_i)_k - (w_i)_{k+1} \neq \pm 1 \bmod n$. In other words, there exists $u \in D_\J$ such that $u \precdot w_i$ and $L_{w_0}(u) = L_{w_0}(w) \tilde{s}_k$.
Let $S_2$ be the set of $\tilde{s}_k$ for $k \in [n]$ such that $L_{w_0}(w_i) \tilde{s}_k < L_{w_0}(w_i)$ but $(w_i)_k - (w_i)_{k+1} \equiv \pm 1 \bmod n$.
Let $S_3$ be the set of $\tilde{s}_k$ for $k \in [n]$ such that $L_{w_0}(w_i) \tilde{s}_k > L_{w_0}(w_i)$.
In particular, $S_1 \sqcup S_2 \sqcup S_3 = \{\tilde{s}_1, \dots, \tilde{s}_n\}$.

Let $I_1 \to I_2 \to \cdots I_n \to I_1$ be the circuit of $w_i$. For each $\tilde{s}_t \in S_1$, let $F_t$ be the facet $\conv(e_{I_k} \mid k \neq t)$ of $\Delta_{(w_i)}$.
Let $\mathcal{F} = \{F_t \mid \tilde{s}_t \in S_1\}$.
In other words, $\mathcal{F} = \{ \Delta_{(w_i)} \cap \Delta_{(u)} \mid L_{w_0}(u) = L_{w_0}(w) \text{ for some } \tilde{s} \in S_1 \}$.

We want to show that $\Delta_{(w_i)} \cap (\Delta_{(w_0)} \cup \cdots \cup \Delta_{(w_{i-1})}) \subseteq \bigcup \mathcal{F}$.
The reverse inclusion $\supseteq$ is immediate.
To show $\subseteq$, we want to show that, for any $j \in [0,i-1]$, $\Delta_{(w_i)} \cap \Delta_{(w_j)} \subseteq \bigcup \mathcal{F}$.

For any $j \in [0,i-1]$, if $w_j \prec w_i$, then by definition, there is a shortest path $w_0 \to \cdots \to w_j \to \cdots \to u \to w_i$. Then $\Delta_{(w_j)} \cap \Delta_{(w_i)} \subset \Delta_{(u)} \cap \Delta_{(w_i)} \in \mathcal{F}$.

For any $j \in [0,i-1]$, if $w_j$ and $w_i$ are not comparable, we consider the minimal set $J$ of simple reflections that generate $L_{w_i}(w_j)$.

Suppose $J \cap S_1 \neq \emptyset$,
we can find a reduced word for $L_{w_j}(w_i)$ of the form $u^{S_1} u_{S_1}$ such that $u_{S_1}$ is generated by $S_1$ \cite[Prop. 2.4.4]{bjornerbrenti}. 
If $u_{S_1} \neq 1$, then let $\tilde{s}_t$ be the last reflection in $u_{S_1}$. We have $\Delta_{(w_i)} \cap \Delta_{(w_j)} \subseteq F_t \subseteq \bigcup \mathcal{F}$ by \cref{lem:facets}.

Suppose $J \cap S_1 = \emptyset$ but $J \cap S_2 \neq \emptyset$, then we can find a reduced word for $L_{w_j}(w_i)$ of the form $u^{S_2}u_{S_2}$ where $u_{S_2}$ is generated by $S_2$. However, $L_{w_0}(w_i)s \notin L_{w_0}(D_\J)$ for $\tilde{s} \in S_2$, so there is no path in $\Gamma_\J$ that corresponds to this reduced word, a contradiction to the second statement of \cref{lem:relabel}.

Suppose $J \subseteq S_3$, that is, $L_{w_i}(w_j)$ can be generated by $S_3$. 
Let $u^{S_3}$ be a reduced word for $L_{w_0}(w_i)$ and let $u_{S_3}$ be a reduced word for $L_{w_i}(w_j)$. We have $u^{S_3} s > u^{S_3}$ for all $s \in S_3$ by the definition of $S_3$.
Moreover, $u^{S_3} u_{S_3}$ is a reduced word for $L_{w_0}(w_j)$ by \cite[Prop. 2.4.4]{bjornerbrenti}.
By \cref{lem:relabel}, there is a shortest path from $w_0$ to $w_j$ passing through $w_i$ so $w_i \prec w_j$, which contradicts our assumption that $j \in [0,i-1]$ and $w_0, \dots, w_s$ is a linear extension of $(\mathcal{P}_{w_0, \J},\prec)$.

Therefore, for any $i \in [2,s]$ and for any $j \in [0,i-1]$, the intersection
between $\Delta_{(w_i)}$ and $\Delta_{(w_j)}$ is a union of facets of $\Delta_{(w_i)}$.
\end{proof}

\begin{rema}
In the spirit of \cite[Thm. 2.1]{bjorner}, \cref{prop:shelling} can be stated in a more general setting, which says that any linear extension of the weak order of the restriction of the Coxeter complex to a convex subset that contains the identity is a shelling. This more general statement will appear in the joint work with Elisabeth Bullock \cite{alcovedhstar}.

If we implement a \emph{breadth first search} on $\Gamma_\J$ from root vertex $w_0$, then the cover of a vertex $w$ in $\mathcal{P}_{w_0,\J}$ is equal to the number of edges that connect $w$ to those previously found vertices.
\end{rema}

\begin{proof}[Proof of \cref{thm:main}]
The statement about the covers follows from \cref{lem:stanley} and \cref{prop:shelling}.
\end{proof}

\begin{exem}
In \cref{fig:hypersimplex25}, we draw $\Gamma_\J$ where $\J = (12,23,34,45,51)$. The positroid polytope associated to $\J$ is the hypersimplex $\Delta_{2,5}$.
Choose $w_0 = 31425$ and consider the poset in \cref{fig:hypersimplex25}.
We have $\cover(31425) = 0, \cover(31245) = \cover(13425) = \cover(23145) = \cover(34125) = \cover(14235) = 1$ and $\cover(12435) = \cover(13245) = \cover(21345) = \cover(23415) = \cover(41235) = 2$,
hence $h^*(\Delta_{2,5}, z) = 1+5z+5z^2$.
\end{exem}

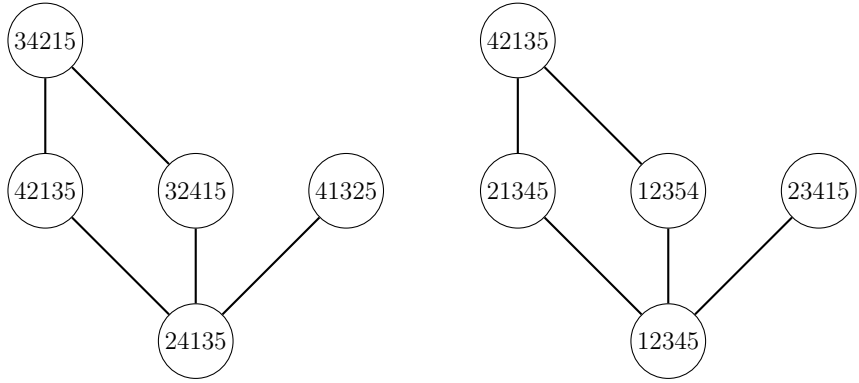
\begin{figure}
    \centering
    \begin{tikzpicture}[every node/.style={draw, circle, inner sep=2pt, scale=.8}]
    \node at (0,0) (2413) {24135};
    \node at (-2,2) (4213) {42135};
    \node at (0,2) (3241) {32415};
    \node at (2,2) (4132) {41325};
    \node at (-2,4) (3421) {34215};
    \draw[thick] (2413) -- (4213);
    \draw[thick] (4213) -- (3421);
    \draw[thick] (2413) -- (3241);
    \draw[thick] (3241) -- (3421);
    \draw[thick] (2413) -- (4132);
    \end{tikzpicture}
    \caption{We show the graph of the circuit triangulation of the positroid polytope associated with the Grassmann necklace $\J = (124, 234, 134, 145, 125)$. 
    The vertices of $\Gamma_\J$ are labeled by $D_\J$.}
    \label{fig:positroid24153}
\end{figure}

\begin{exem}\label{exem:24153}
Consider the connected positroid polytope $P_\J$ for the Grassmann necklace $\J = (124, 234, 134, 145, 125)$. 
We show $\Gamma_\J$ and the Hasse diagram of $\mathcal{P}_{24135,\J}$ in \cref{fig:positroid24153}. 
We have $\cover(24135) = 0, \cover(42135) = \cover(32415) = \cover(41325) = 1$, and $\cover(34215) = 2$.
The $h^*$-polynomial of the positroid polytope associated to the permutation $2413$ is $1+3z+z^2$.
\end{exem}

\section{The $h^*$-polynomial of half-open connected positroid polytopes}\label{sec:half-open}

In this section, we give a combinatorial formula for the $h^*$-polynomial of half-open connected positroid polytopes in terms of descents of permutations.
Then we show how to compute the $h^*$-polynomial of the whole closed polytope by inclusion-exclusion on the half-open ones of smaller dimension.
Our result generalizes results from \cite{NanLi} and \cite{AJR20}.

\begin{defi}
Define the half-open simplex $\nabla_w^\circ$ by 
\begin{align*}
    0 < y_{w_1} & \\
    y_{w_i} \leq y_{w_{i+1}} &\iff w_i < w_{i+1} \\
    y_{w_i} < y_{w_{i+1}} &\iff w_i > w_{i+1} \\
    y_{w_n} \le 1 &
\end{align*}
\end{defi}

\begin{lemm}[\cite{AJR20}, Lem. 4.5]
The half-open simplices are mutually disjoint, and they triangulate the half open hypercube $(0,1]^n$ for $n \geq 1$. That is,
$$(0,1]^n = \bigsqcup_{w \in S_n} \nabla_w^\circ. $$
\end{lemm}

\begin{lemm}[\cite{multihyper}]\label{lem:HJV}
The $h^*$-polynomial of the half-open simplex $\nabla_w^\circ$ is $h^*(\nabla_w^\circ, z) = z^{\des(w)+1}$.
\end{lemm}

\begin{exem}
    Let $w = 3241$, then $\nabla_w^\circ$ is defined by
    \[
    0 < y_3 < y_2 \leq y_4 < y_1 \le 1.
    \]
\end{exem}

\begin{rema}
In \cite{multihyper}, the half-open hypersimplex they consider for 3241 is defined by
\[
0 \le y_3 < y_2 \leq y_4 < y_1 < 1.
\]
In \cite{AJR20}, these half-open hypersimplices are mutually disjoint and triangulate $[0,1)^n$.
We modify these conventions because \cite{parisi2024magic} and \cite{alcove1} both identify a permutation $w \in S_{n-1}$ with $w_1\dots w_{n-1} n \in S_n$ while \cite{AJR20} identifies $w \in S_{n-1}$ with $0 w_1\dots w_{n-1} \in S_n$.
\end{rema}

\begin{defi}
Let $w \in D_n$. We denote $\phi(\nabla_{\underline{w}}) \subset [0,1)^{n-1}$ by $\Delta_{(w)}^\circ$ and call it the \emph{half-open $(w)$-simplex}.
\end{defi}

\begin{rema}
The facets of these $(w)$-simplices and positroid polytopes are all of the form $x_{[i,j]} = k$ for some $i, j \in [n]$ and $k \in \Z$.
We will call a facet of a positroid polytope or $(w)$-simplex \emph{upper} if it is of the form $x_{[i,j]} = k$ such that the positroid polytope satisfies $x_{[i,j]} \leq k$.
\end{rema}

The next corollary follows directly from the definition of $\phi$.
\begin{coro}\label{coro:half-open-w-simplex}
Let $w \in D_n$. The half-open $(w)$-simplex is equal to $\Delta_{(w)}$ with all upper facets removed.
\end{coro}

\begin{thm}\label{thm:half-open}
Let $P_\J$ be a connected positroid polytope, where $\J$ is the associated Grassmann necklace.
Consider the half-open positroid polytope $\tilde{P}_\J \subset [0,1)^{n-1}$ which is the projection of $P_\J$ onto the first $(n-1)$ coordinates with all upper facets removed.
Then the $h^*$-polynomial of $\tilde{P}_\J$ is equal to
$h^*(\tilde{P}_\J,z) = \sum_{w \in D_\J} z^{\des(\underline{w})+1}$.
\end{thm}

\begin{proof}
By \cref{coro:half-open-w-simplex}, we have that the half open polytope $\tilde{P}_\J$ is triangulated by $\Delta_{(w)}^\circ$.
By \cref{lem:HJV} and the additivity of $h^*$-polynomials, we have that $h^*(\tilde{P}_\J,z) = \sum_{w \in D_\J} z^{\des(\underline{w})+1}$.
\end{proof}

\begin{exem}
Consider the Grassmann necklace $\J = (124, 234, 134, 145, 125)$.
The half-open positroid polytope $\tilde{P}_\J$ associated to $\J$ has facet-defining inequalities
$x_1 < 1, x_2 < 1, x_4 < 1, x_3 \geq 0$ and
\begin{align*}
    x_1+x_2+x_3 &< 2 \\
    x_1+x_2 &\ge 1 \\
    x_1+x_2+x_3+x_4 &\ge 2
\end{align*}
\end{exem}

To compute the $h^*$-polynomial of a polytope from the $h^*$-polynomials of the half-open polytope and its faces, we use the inclusion-exclusion principle.
Let $P$ be a polytope and let $F_1, \dots, F_\ell$ be a collection of facets of $P$.
Consider the restriction of the face poset of $P$ to have coatoms $F_1, \dots, F_\ell$.
This poset $\mathcal{P}_{F_1, \dots, F_\ell}$ describes all the faces of $P$ in the intersections of $F_1, \dots, F_\ell$.
Let $\mu_{F_1, \dots, F_\ell}$ be the M\"obius function of this poset.

The next proposition follows from inclusion-exclusion on the face poset and additivity of Ehrhart polynomials.
\begin{prop}\label{prop:from_half-open}
Let $P$ be a polytope. 
Let $F_1, \dots, F_\ell$ be a collection of facets of $P$, and 
let $\tilde{P} = P \setminus (F_1 \cup \cdots \cup F_\ell)$.
The $h^*$-polynomial of the polytope $P$ is equal to
$$h^*(P,z) = h^*(\tilde{P},z) - \sum_{F \in \mathcal{P}_{F_1, \dots, F_\ell}(F), F \neq P} \mu_{F_1, \dots, F_\ell}(F,P) (1-z)^{\dim(P) - \dim(F)} h^*(F,z).$$
\end{prop}

\begin{figure}
    \centering
    \begin{tikzpicture}
    \coordinate (A) at (0,0);
    \coordinate (B) at (4,0);
    \coordinate (C) at (4,2);
    \coordinate (D) at (0,2);

    \coordinate (Apex) at (2,5);

    \draw[thick] (A) -- (B) -- (C);
    \draw[thick] (C) -- (D) --(A);

    \draw[thick] (Apex) -- (A);
    \draw[thick] (Apex) -- (B);
    \draw[thick] (Apex) -- (C);
    \draw[thick] (Apex) -- (D);

    \node[below left] at (A) {0110};
    \node[below right] at (B) {1010};
    \node[above right] at (C) {1001};
    \node[above left] at (D) {0101};

    \draw[fill=red!80,opacity=0.5] (Apex) -- (B) -- (C) -- cycle;
    \draw[fill=yellow!80,opacity=0.5] (Apex) -- (A) -- (D) -- cycle;
    \draw[fill=blue!80,opacity=0.5] (Apex) -- (A) -- (B) -- cycle;

    \node[above] at (Apex) {1100};
    \end{tikzpicture}
    \caption{The positroid polytope associated to the Grassmann necklace $\J = (12,23,13,14)$ is a pyramid. The red facet corresponds to $F_1: x_1 = 1$; the yellow facet corresponds to $F_2: x_2 = 1$; the blue facet corresponds to $F_3: x_1+x_2+x_3 = 2$. These are all the upper facets of this positroid polytope.}
    \label{fig:2413}
\end{figure}
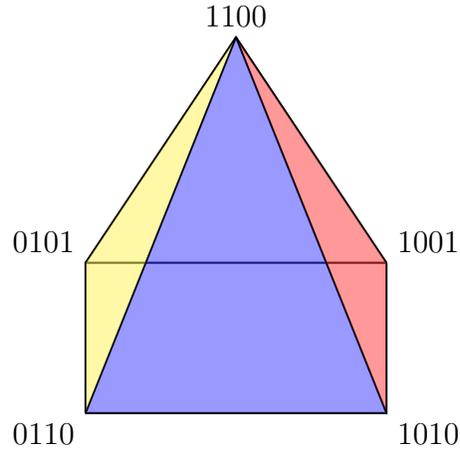

\begin{figure}
\centering
    \begin{tikzpicture}
    \node (P) at (2,7) {\pyramid};
    \node (F1) at (0,5)  {$\triangle$};
    \node (F3) at (2,5) {$\triangle$};
    \node (F2) at (4,5) {$\triangle$};
    \draw[thick] (P) -- (F1);
    \draw[thick] (P) -- (F2);
    \draw[thick] (P) -- (F3);
    \node (E1) at (1,3) {/};
    \node (E2) at (3,3) {$\backslash$};
    \draw[thick] (F1) -- (E1) -- (F3);
    \draw[thick] (F3) -- (E2) -- (F2);
    \node (Apex) at (2,1) {1100};
    \draw[thick] (E1) -- (Apex) -- (E2);
    \end{tikzpicture}
    \hspace{3cm}
    \begin{tikzpicture}
    \node (P) at (2,7) {1};
    \node (F1) at (0,5)  {-1};
    \node (F3) at (2,5) {-1};
    \node (F2) at (4,5) {-1};
    \draw[thick] (P) -- (F1);
    \draw[thick] (P) -- (F2);
    \draw[thick] (P) -- (F3);
    \node (E1) at (1,3) {1};
    \node (E2) at (3,3) {1};
    \draw[thick] (F1) -- (E1) -- (F3);
    \draw[thick] (F3) -- (E2) -- (F2);
    \node (Apex) at (2,1) {0};
    \draw[thick] (E1) -- (Apex) -- (E2);
    \end{tikzpicture}
    \caption{The poset $\mathcal{P}_{F_1,F_2,F_3}$ and the value of its M\"obius function $\mu_{F_1,F_2,F_3}(-,P_\J)$.}
    \label{fig:2413poset}
\end{figure}
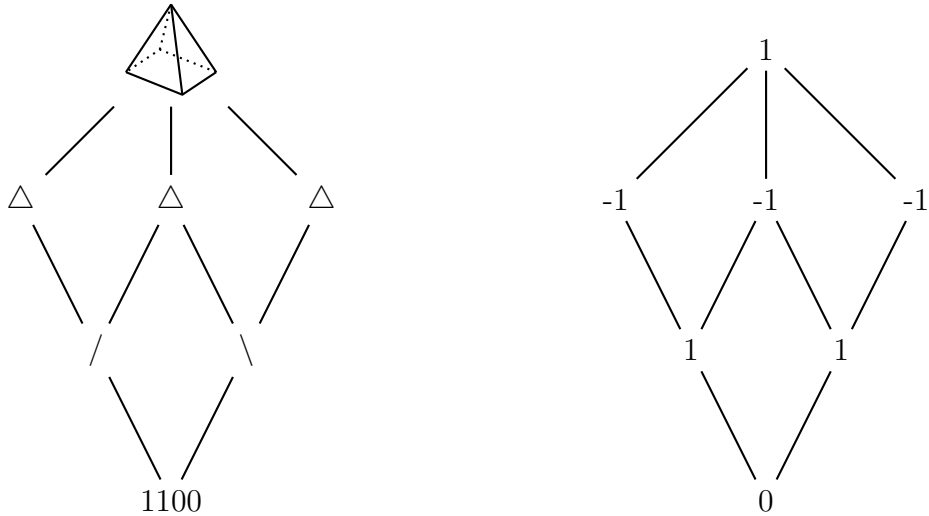

\begin{exem}\label{exem:2413}
Consider the Grassmann necklace $\J = (12,23,13,14)$.
Then $D_\J$ consists of permutations in $S_4$ that end with 4 such that $\cdes_L(w|_{[3,4]}) \leq 1$ so $D_\J = \{1324, 2134\}$.
Now $\des(132)=\des(213) = 1$, so $h^*(\tilde{P}_\J,z) = 2z^2$.
To compute the $h^*$-polynomial of the whole positroid polytope $P$, which is a pyramid (see \cref{fig:2413}), we use \cref{prop:from_half-open}.
The upper facets we removed are
\begin{align*}
    F_1: x_1 &= 1 \\
    F_2: x_2 &= 1 \\
    F_3: x_1+x_2+x_3 &= 2.
\end{align*}
They are all triangles and $F_1 \cap F_2 = F_1 \cap F_2 \cap F_3$ is the apex of the pyramid. Also $F_1 \cap F_3$ and $F_2 \cap F_3$ are both segments.
Both triangles and segments have $h^*$-polynomial 1.
Therefore, by inclusion-exclusion, the $h^*$-polynomial of $P_\J$ is
$h^*(P_\J,z) = 2z^2 + 3(1-z) - 2(1-z)^2 = 1+z$.
\end{exem}

\begin{figure}
    \centering
    \tdplotsetmaincoords{60}{120}
    \begin{tikzpicture}[tdplot_main_coords, scale=3]

    \coordinate (A1) at (0, 0, 0);    
    \coordinate (B1) at (1.5, 0, 0);    
    \coordinate (C1) at (1.5, 1, 0);    
    \coordinate (D1) at (0, 1, 0);    
    
    \coordinate (A2) at (0, 0.5, 1);  
    \coordinate (B2) at (1.5, 0.5, 1);  

    \draw[thick,dotted] (A1) -- (B1);
    \draw[thick,dotted] (A1) -- (A2);
    \draw[thick,dotted] (A1) -- (D1);

    \draw[thick] (B1) -- (C1) -- (B2) -- cycle;  
    \draw[thick] (D1) -- (A2) -- (B2) -- (C1) -- cycle;  

    \node at (A1) [below] {$01110$};
    \node at (B1) [left] {$01101$};
    \node at (C1) [below right] {$10101$};
    \node at (D1) [right] {$10110$};
    
    \node at (A2) [above] {$11010$};
    \node at (B2) [left] {$11001$};

    \end{tikzpicture}
    \caption{The facet defined by $x_1+x_2+x_3=2$ of the positroid polytope associated to the Grassmann necklace $(124, 234, 134, 145, 125)$. This is a prism, the product of a triangle and a segment.}
    \label{fig:product}
\end{figure}
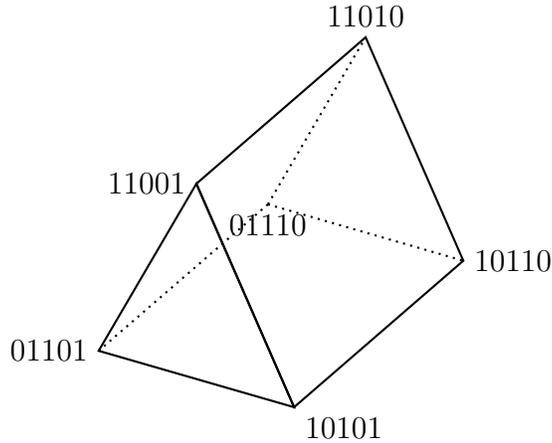

\begin{exem}
Consider the positroid with bases \\$\{\{1, 2, 4\},
 \{1, 2, 5\},
 \{1, 3, 4\},
 \{1, 3, 5\},
 \{1, 4, 5\},
 \{2, 3, 4\},
 \{2, 3, 5\},
 \{2, 4, 5\}\}$, associated to the Grassmann necklace $\J = (124, 234, 134, 145, 125)$.
We have $D_\J = \{34215, 42135, 24135,32415, 41325\}$.
Therefore $h^*(\tilde{P}_\J,z) = z^2+4z^3$.
To compute the $h^*$-polynomial of the whole positroid polytope $P_\J$, we use \cref{prop:from_half-open}.
The upper facets we removed are
\begin{align*}
    F_1 : x_1 &= 1 \\
    F_2 : x_2 &= 1 \\
    F_3: x_4 &= 1 \\
    F_4: x_1+x_2+x_3 &= 2
\end{align*}
The first three facets are isomorphic to the pyramid. The last facet is a prism, which is a product of a triangle and a segment; see \cref{fig:product}. The Ehrhart polynomial of the triangle is equal to $\binom{t+2}{2}$ and the Ehrhart polynomial of the segment is equal to $1+t$. 
Thus, the Ehrhart polynomial of the prism is equal to $\binom{t+2}{2}(1+t)$ by \cref{rema:reduce_to_connected} and the $h^*$-polynomial of the prism is equal to $1+2z$.

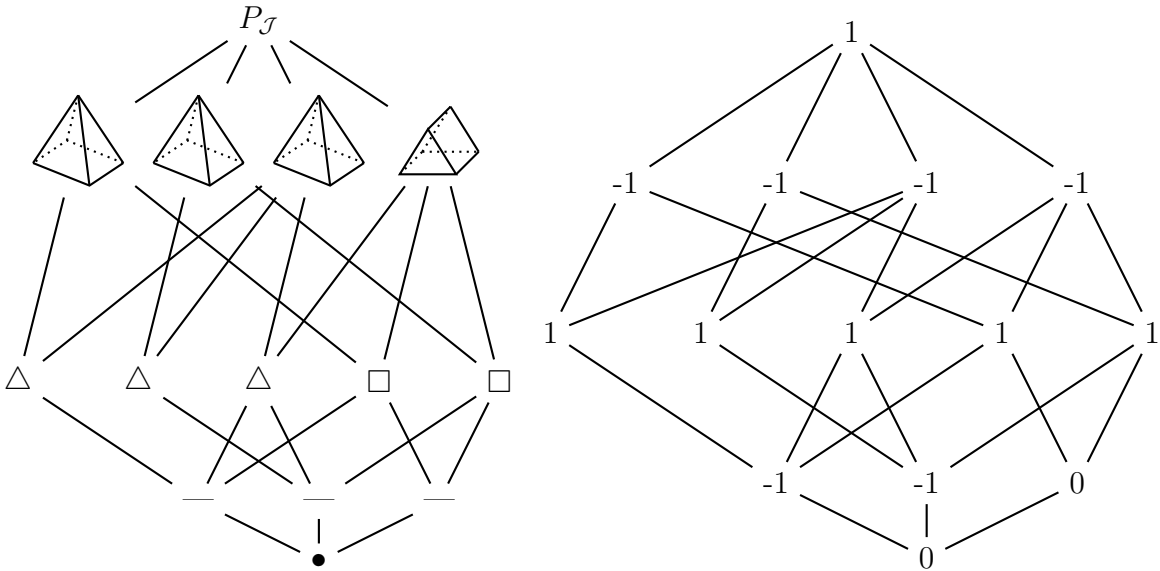
\begin{figure}
    \centering
\begin{tikzpicture}[scale=.8]
\node (P) at (0,6) {$P_\J$};
\node (F1) at (-3,4) {\pyramid};
\node (F2) at (-1,4) {\pyramid};
\node (F3) at (1,4) {\pyramid};
\node (F4) at (3,4) {\prism};

\foreach \i in {1,2,3,4}
{\draw[thick] (P) -- (F\i);}

\node (F5) at (-4,0) {$\triangle$};
\node (F6) at (-2,0) {$\triangle$};
\node (F7) at (0,0) {$\triangle$};
\node (F8) at (2,0) {$\square$};
\node (F9) at (4,0) {$\square$};

\draw[thick] (F1) -- (F5) -- (F3);
\draw[thick] (F2) -- (F6) -- (F3);
\draw[thick] (F3) -- (F7) -- (F4);

\draw[thick] (F1) -- (F8) -- (F4);
\draw[thick] (F2) -- (F9) -- (F4);

\node (P1) at (3,-2) {|};
\draw[thick] (F8) -- (P1) -- (F9);
\node (P2) at (-1,-2) {|};
\draw[thick] (F5) -- (P2) -- (F8);
\draw[thick] (F7) -- (P2);
\node (P3) at (1,-2) {|};
\draw[thick] (F6) -- (P3) -- (F7);
\draw[thick] (F9) -- (P3);

\node (A) at (1,-3) {$\bullet$};
\draw[thick] (P1) -- (A);
\draw[thick] (P2) -- (A);
\draw[thick] (P3) -- (A);
\end{tikzpicture}
\begin{tikzpicture}
\node (P) at (0,6) {1};
\node (F1) at (-3,4) {-1};
\node (F2) at (-1,4) {-1};
\node (F3) at (1,4) {-1};
\node (F4) at (3,4) {-1};

\foreach \i in {1,2,3,4}
{\draw[thick] (P) -- (F\i);}

\node (F5) at (-4,2) {1};
\node (F6) at (-2,2) {1};
\node (F7) at (0,2) {1};
\node (F8) at (2,2) {1};
\node (F9) at (4,2) {1};

\draw[thick] (F1) -- (F5) -- (F3);
\draw[thick] (F2) -- (F6) -- (F3);
\draw[thick] (F3) -- (F7) -- (F4);

\draw[thick] (F1) -- (F8) -- (F4);
\draw[thick] (F2) -- (F9) -- (F4);

\node (P1) at (3,0) {0};
\node (P2) at (-1,0) {-1};
\node (P3) at (1,0) {-1};

\draw[thick] (F8) -- (P1) -- (F9);
\draw[thick] (F5) -- (P2) -- (F8);
\draw[thick] (F7) -- (P2);
\draw[thick] (F6) -- (P3) -- (F7);
\draw[thick] (F9) -- (P3);

\node (A) at (1,-1) {0};
\foreach \i in {1,2,3}
{\draw[thick] (A) -- (P\i);}
\end{tikzpicture}
\caption{The poset $\mathcal{P}_{F_1,\dots,F_4}$ and the value of its M\"obius function $\mu_{F_1, \dots, F_4}(-,P_\J)$.}
\label{fig:24153poset}
\end{figure}

The intersections $F_1 \cap F_3, F_2 \cap F_3, F_3 \cap F_4$ are triangles, and the intersections $F_1 \cap F_4, F_2 \cap F_4$ are squares, which has $h^*$-polynomial $1+z$.
We have $F_1 \cap F_2 = F_1 \cap F_2 \cap F_4$ is a segment and $F_1 \cap F_2 \cap F_3 = F_1 \cap F_2 \cap F_3 \cap F_4$ is a point.
All the other triple intersections are segments.
The $h^*$-polynomial of $P$ is therefore
\begin{align*}
    h^*(P_\J,z) &= z^2+4z^3 + (3(1+z)+1+2z)(1-z) - (3+2(1+z))(1-z)^2 + 2(1-z)^3 \\
    &= 1+3z+z^2.
\end{align*}
The result coincides with \cref{exem:24153}.
\end{exem}

\section{Tree positroids}\label{sec:tree}
When the plabic graph of a positroid is acyclic, we call it a \emph{tree positroid} \cref{appendix}. 
In this section, we apply \cref{thm:main} to the special case of a tree positroid.

To each plabic graph, one can associate a \emph{plabic tiling} \cref{appendix} \cite{plabictiling}.
Tree positroids are those positroids whose plabic tilings are \emph{bicolored subdivision}, as defined in \cite{parisi2024magic}. We follow the conventions of \cite{parisi2024magic}.

\begin{defi}
Let $\P_n$ be a convex $n$-gon with vertices labeled from 1 to $n$ in clockwise order. 
A \emph{bicolored subdivision} $\tau$ is a partition of $\P_n$ into black and white polygons such that two polygons sharing an edge have different colors. 
We say that $\tau$ has \emph{type} $(k,n)$ if any triangulation of the black polygons consists of exactly $k$ black triangles.
\end{defi}

\begin{defi}\label{def:arc}
Let $\tau$ be a bicolored subdivision. Given a pair of vertices $i,j$ of $\mathbf{P}_n$, we say that the arc $i \to j$ is 
\begin{itemize}
    \item \emph{compatible} with $\tau$ if the arc either is an edge of a black or white polygon, or lies entirely inside a black or white polygon of $\tau$,
    \item \emph{facet-defining} if it bounds a black polygon of $\tau$ to its left.
\end{itemize} If $i \to j$ is compatible with $\tau$, the \emph{area to the left of $i \to j$}, denoted by $\area(i \to j)$, is the number of black triangles to the left of $i\to j$ in any triangulation of the black polygons of $\tau$.
\end{defi}

\begin{figure}
    \centering
    \begin{tikzpicture}
        \draw node (1) at (0,2) {1};
        \draw node (2) at (2,0) {2};
        \node[inner sep=0pt, circle, fill=black, minimum size = 7pt] (b) at (0,0) {};
        \node[inner sep=0pt, circle, draw = black, minimum size=7pt] (w) at (-1,-1) {};
        \draw node (4) at (-3,-1) {4};
        \draw node (3) at (-1,-3) {3};
        \draw[thick] (1)--(b);
        \draw[thick] (2)--(b);
        \draw[thick] (b)--(w);
        \draw[thick] (w)--(3);
        \draw[thick] (w)--(4);
        \draw[fill = black, opacity = 0.5] (1,1) -- (1,-2) -- (-2,1) -- cycle;
        \draw (1,-2) -- (-2,-2) -- (-2,1);
        \draw node at (-2.2,1.2) {1};
        \draw node at (1.2,1.2) {2};
        \draw node at (1.2,-2.2) {3};
        \draw node at (-2.2,-2.2) {4};
    \end{tikzpicture}
    \caption{The tree plabic graph and bicolored subdivision associated to the Grassmann necklace $\J = (12, 23, 13, 14)$.}
    \label{fig:tree2413}
\end{figure}
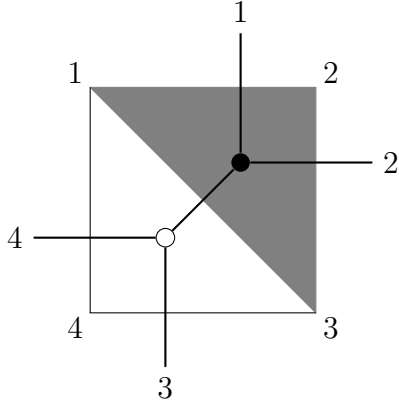

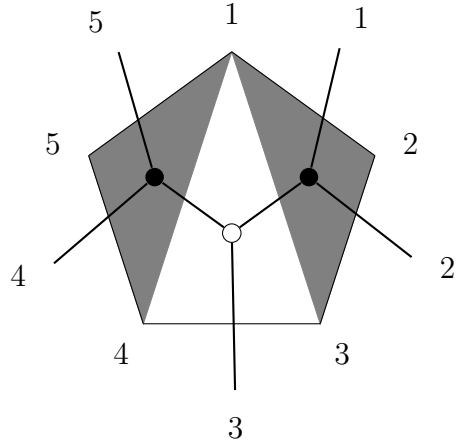
\begin{figure}
    \centering
    \begin{tikzpicture}[dot/.style={circle,inner sep=1pt,fill,label={#1},name=#1},
  extended line/.style={shorten <=-#1},
  extended line/.default=1cm]
    \def\radius{2}
    
    \foreach \i in {1,2,3,4,5} {
        \coordinate (P\i) at ({90 - 360/5 * (\i - 1)}:\radius);
    }
    
    \foreach \i in {1,2,3,4,5} {
        \pgfmathtruncatemacro{\j}{mod(\i,5) + 1}
        \draw (P\i) -- (P\j);
    }

    \foreach \i in {1,2,3,4,5} {
        \node at ({90 - 360/5 * (\i - 1)}: \radius + 0.5) {\i};
    }
    
    \draw[fill = black, opacity = 0.5] (P1) -- (P2) -- (P3) -- (P1);
    \draw[fill = black, opacity = 0.5] (P1) -- (P4) -- (P5) -- (P1);

    \node[inner sep=0pt, circle, fill=black, minimum size = 7pt] (b1) at (barycentric cs:P1=1,P2=1,P3=1) {};
    \node[inner sep=0pt, circle, fill=black, minimum size = 7pt] (b2) at (barycentric cs:P1=1,P4=1,P5=1) {};
    \node[inner sep=0pt, circle, draw = black, minimum size=7pt] (w) at (barycentric cs:P1=1,P3=1,P4=1) {};

    \def\radius{2.5}
    
    \foreach \i in {1,2,3,4,5} {
        \coordinate (Q\i) at ({55-360/5 * (\i - 1)}:\radius);
    }
    \foreach \i in {1,2,3,4,5} {
        \node at ({55-360/5 * (\i - 1)}: \radius + 0.5) {\i};
    }

    \draw[thick] (Q1) -- (b1) -- (Q2);
    \draw[thick] (b1) -- (w) -- (b2);
    \draw[thick] (w) -- (Q3);
    \draw[thick] (Q5) -- (b2) -- (Q4);
\end{tikzpicture}
    \caption{The tree plabic graph and bicolored subdivision associated to the Grassmann necklace $\J = (124, 234, 134, 145, 125)$.}
    \label{fig:tree24153}
\end{figure}

\begin{exem}
In \cref{fig:tree2413}, the arc $1 \to 3$ is facet defining, but $2\to4$ is not.
\end{exem}

The tree positroid polytopes are triangulated by $(w)$-simplices where $w$ extends a \emph{partial cyclic order} \cite{parisi2024magic}.

\begin{defi}\label{def:cyclic-order}
    A \emph{(partial) cyclic order} on a finite set $X$ is a ternary relation $C \subset \binom{X}{3}$ such that for all $a, b, c, d \in X$:
    \begin{align}
        (a,b,c) \in C &\implies (c, a, b) \in C \tag{cyclicity}\\
        (a,b,c) \in C &\implies (c, b, a) \notin C \tag{asymmetry}\\
        (a,b,c) \in C \text{ and }(a, c, d) \in C &\implies (a,b,d) \in C \tag{transitivity}
    \end{align}
    A cyclic order $C$ is \emph{total} if for all $a, b, c\in X$, either $(a,b,c)\in C$ or $(a, c, b) \in C$.
    A cyclic order $C'$ is a \emph{circular extension} of $C$ if $C \subseteq C'$.
    We denote the set of all circular extensions of $C$ by $\Ext(C)$.
  \end{defi}

Informally, a total cyclic order $C$ on $[n]$ is a way of placing $1, \dots, n$ on a circle, just as a total order is a way of placing $1, \dots, n$ on a line.

\begin{defi}
Let $w=w_1\dots w_n \in S_n$. The $w$-order $C_w$ is the total cyclic order obtained by placing $w_1, w_2, \dots, w_n$ on the circle clockwise. We identify this total cyclic order with the $n$-cycle $(w)$ and so may write $(w)$ for $C_w$ or write $C_w=(w_1\ w_2 \dots w_n)$.
\end{defi}  

Note that each total cyclic order on $[n]$ is of the form $C_w$ for a unique permutation $w \in D_n$ (cf. \cref{Dk+1n} for the definition of $D_n$). We move interchangeably between $w \in D_n$, the $n$-cycle $(w)$ and the total cyclic order $C_w$.

\begin{defi}
Let $x_1,\dots,x_m$ be a sequence of $m$ distinct elements of $[n]$ (for $3 \leq m \leq n$).
We let $C=C_{(x_1, x_2 \dots, x_m)}$ denote the partial cyclic order on $[n]$ in which 
for each triple $1\leq i < j < \ell \leq m$ 
        we have $(x_i,x_j, x_{\ell}) \in C$ (which implies by cyclicity that
        also $(x_j,x_{\ell}, x_i)$ and $(x_{\ell}, x_i, x_j)$ lie in $C$).  
We call this partial cyclic order a \emph{chain}.
\end{defi}

\begin{defi}\label{def:cyclic-from-perm-subdiv}
Let $\tau$ be a bicolored subdivision of $\mathbf{P}_n$ with $q$ polygons $P_1, \dots, P_q$ which are black or white.
If $P_a$ is white (respectively, black), we let $v_1,\dots,v_r$ denote its list of vertices read in clockwise (respectively, counterclockwise) order.  We then associate the chain $C_a = C_{(v_1,\dots,v_r)}$ to $P_a$.  Finally we define the \emph{$\tau$-order}
to be the partial cyclic order which is the union of the partial cyclic orders associated
to the black and white polygons: 
    \[C_{\tau}:=C_1 \cup \dots \cup C_q.\]
\end{defi}

Not all cyclic orders have a circular extension \cite{Megiddo1976}, that is, 
$\Ext(C)$ could be empty. 
Moreover, the problem of determining
whether a cyclic order has a circular extension is NP-complete
\cite{Megiddo1976}.

\begin{thm}[\cite{m=2hypersimplex}, Thm. 9.2, Prop. 9.5, Prop. 9.6]
If $\tau$ is a bicolored subdivision, then the corresponding positroid polytope $\Gamma_\tau$ is cut out of $\R^n$ by the equality $x_1+\cdots+x_n = k$ and
$$\area(i \to j) \le x_i + \cdots x_{j-1} \le \area(i \to j)+1 \quad \text{for any compatible arc } i \to j \text{ of } \tau $$
or, alternatively, by the following facet inequalities
\begin{itemize}
    \item $x_i \ge 0$ if there is a white polygon of $\tau$ with vertex $i$;
    \item $x_i + \cdots + x_{j-1} \ge \area(i \to j)$, if $i \to j$ is a facet-defining arc of $\tau$.
\end{itemize}
\end{thm}

\begin{prop}[Cor. 4.8, \cite{parisi2024magic}]
Let $\sigma$ be a bicolored subdivision of type $(k,n)$. Then
$$\Gamma_\sigma = \bigcup_{(w) \in \Ext(C_\sigma)} \Delta_{(w)}.$$
That is, $\Gamma_\sigma$ is the union of $(w)$-simplices $\Delta_{(w)}$.
\end{prop}

\begin{exem}
In \cref{fig:tree2413}, we have chains $(3,2,1)$ and $(1,3,4)$. The only two cyclic total order on $[4]$ with these two chains are $2134$ and $1324$, which coincide with our computation in \cref{exem:2413}.
\end{exem}

\begin{exem}
In \cref{fig:tree24153}, we have that $C_\tau$ consists of the chains $(3,2,1), (1,3,4)$ and $(5,4,1)$. We can check that $D_\J$ in \cref{exem:24153} is exactly the set of circular extensions of $C_\tau$.
\end{exem}

Analogous to \cite[Theorem 2.9]{alcove1}, we can characterize the edges in the graph of the circuit triangulation of tree positroid polytopes via the following corollary to \cref{thm:main}.

\begin{coro}\label{coro:main}
Let $P_\tau$ be a tree positroid polytope associated with the bicolored subdivision $\tau$.
Let $\Gamma_\tau$ be the graph of the circuit triangulation of $P_\tau$.
For any $w_0 \in \Ext(C_\tau)$, let $(\mathcal{P}_{w_0,\tau}, \prec)$ be the corresponding poset on $\Ext(C_\tau)$.
The cover statistic of $\mathcal{P}_{w_0,\tau}$ gives the $h^*$-polynomial of $P_\tau$, i.e.,
$h^*(P_\tau, z) = \sum_{w \in \Ext(C_\tau)} z^{\cover(w)}$ where
$\cover(w) = \#\{u \in \Ext(C_\tau) \mid u \precdot w\}$ is the number of elements covered by $w$ in $\mathcal{P}_{w_0,\tau}$.
\end{coro}

\section{Positroid tiling of positroid polytopes}
Tilings of the hypersimplex are connected to tilings of the amplituhedron, and can be used to compute the scattering amplitudes.
It was known that positroids tiles of the hypersimplex are in bijection with tree positroids.
We study the structure of the positroid tilings of positroid polytopes whose plabic graph contains cycles.

\begin{defi}[\cite{wil23}]\label{def:tiling0} 
Let $\phi: \Grk \to X$ be 
a continuous surjective map where $\dim X=d$. The image ${\phi(S)}$ of a positroid cell $S \subset \Grk$ is an \emph{open tile} for $X$ (with respect to $\phi$) if $\phi$ is injective on $S$ and $\dim S =d$. The closure $\overline{\phi(S)}$ of an open tile is called a \emph{tile} for $X$.
 
A \emph{positroid tiling} of $X$ (with respect to $\phi$) is a collection $\{\overline{\phi(S)}\}_{S \in \mathcal{C}}$ of tiles 
satisfying:
\begin{itemize}
	\item (disjointness): open tiles $\phi(S)$
		and $\phi(S')$ for distinct positroid cells $S$ and $S'$ in $\mathcal{C}$ are disjoint
              \item (covering): $\cup_{S \in \mathcal{C}}\overline{\phi(S)} = X$.
	\end{itemize}
\end{defi}

\begin{thm}[Prop 3.15, 3.16, \cite{LPW}]
The positroid tiles for $\Delta_{k+1,n}$ are in bijection with tree positroids.
\end{thm}

\begin{prop}[Prop. 5.14, \cite{parisi2024magic}]
Let $\Gamma_M \subset \Delta_{k+1,n}$ be a positroid polytope. Then there is a positroid subdivision of $\Delta_{k+1,n}$ that contains $\Gamma_M$. Moreover, each full-dimensional positroid polytope admits a positroid tiling.
\end{prop}

\begin{coro}[Cor. 5.15, \cite{parisi2024magic}]
Let $\Gamma_M \subset \Delta_{k+1,n}$ be a full-dimensional positroid polytope. Then $\Gamma_M$ admits a positroid tiling, and all positroid tilings of $\Gamma_M$ have the same cardinality.
\end{coro}

We can reduce the problem to connected reduced bipartite plabic graphs that are unions of pairwise adjacent cycles. Two cycles are adjacent if they share at least an edge.

\begin{prop}\label{prop:reduction}
Let $M$ be a connected positroid on $[n]$ with reduced bipartite plabic graph $G$ associated with the plabic tiling $\T(G)$. 
For all $i, j \in [n]$, we draw an arc $i \to j$ if $i \to j$ is compatible with $\T(G)$ (see \cref{def:arc}).
Let $\T_1,\dots,\T_\ell$ be the sub-plabic-tilings of $\T(G)$ cut out by compatible arcs that are not bicolored subdivisions.
Then the cardinality $\N(M) =: \N(\T(G))$ of any positroid tiling of $\Gamma_M$ is equal to the product
$$ \N(\T(G)) = \prod_{i=1}^\ell \N(\T_i). $$
Moreover, for $i = 1,\dots,\ell$, each $\T_i$ is a union of mutually adjacent cycles.
\end{prop}

Postnikov--Speyer--Williams characterized the inclusion relations among positroid polytopes as follows.

\begin{lemm}[Cor. 7.4, \cite{psw}]\label{lem:delete}
Let $M$ be a connected positroid on $[n]$ with reduced bipartite plabic graph $G$.
The new plabic graph $G'$ obtained from $G$ by removing an edge from $G$ defines a positroid $M'$ such that $\Gamma_{M'} \subset \Gamma_M$.
Moreover, any positroid $M'$ with $\Gamma_{M'} \subset \Gamma_M$ can be obtained this way.
\end{lemm}

\cref{prop:reduction} and \cref{lem:delete} give us an algorithm to compute the positroid tilings of any positroid polytope by induction on the number of cycles.

\appendix
\section{decorated permutations, plabic graphs}\label{appendix}
\subsection{Decorated permutations}

There is a bijection between Grassmann necklaces and decorated permutations \cite{postnikov2006}.

\begin{defi}
A \emph{decorated permutation} $\pi^{:} = (\pi, \col)$ is a permutation $\pi \in S_n$ together with a coloring function $\col$ to black or white from the set of fixed points $\{i \mid \pi(i) = i\}$.
\end{defi}

Each plabic graph with $n$ boundary vertices defines the \emph{trip permutation} $\pi_G \in S_n$ such that $\pi_G(i) = j$ whenever the trip that starts at the boundary vertex $i$ ends at boundary vertex $j$.
If $i$ is a fixed point of $\pi_G$, then we color $i$ by the color of the $i$-th boundary vertex.

There is a bijection between decorated permutations and Grassmann necklaces.
Given a Grassmann necklace $\J$, we construct the decorated permutation $\pi^:(\J) = (\pi,\col)$ such that
\begin{itemize}
    \item if $J_{i+1} = J_i - \{i\} \cup \{j\}$ for $j \neq i$ then $\pi(i) = j$;
    \item if $J_{i+1} = J_i$ and $i \notin J_i$ then $\pi(i) = i$ is colored in black;
    \item if $J_{i+1} = J_i$ and $i \in J_i$ then $\pi(i) = i$ is colored in white;
\end{itemize}
where the indices $i$ are taken modulo $n$.

\subsection{Plabic graphs}

Equivalence classes of reduced plabic graphs are in bijection with positroids, through its decorated trip permutation.

\begin{defi}\label{def:plabic}
A {\it planar bicolored graph} (or ``plabic graph'')
is a planar graph $G$ properly embedded into a closed disk, such that each internal vertex is colored black or white;
each internal vertex is connected by a path to some boundary vertex; 
there are (uncolored) vertices lying on the 
boundary of the disk labeled $1,\dots, n$ for some positive $n$;
and each of the boundary vertices is incident to a single edge.
\end{defi}

There is a natural set of local transformations (moves) of plabic graphs:

(M1) \emph{Square move} (or \emph{urban renewal}).  If a plabic graph has a square formed by
four trivalent vertices whose colors alternate,
then we can switch the
colors of these four vertices.

(M2) \emph{Contracting/expanding a vertex}.
Two adjacent internal vertices of the same color can be merged.
This operation can also be reversed. 

(M3) \emph{Middle vertex insertion/removal}.
We can remove/add degree $2$ vertices.

See \cref{fig:M1} for depictions of these three moves.

\begin{figure}[h]
\centering
\includegraphics[height=.5in]{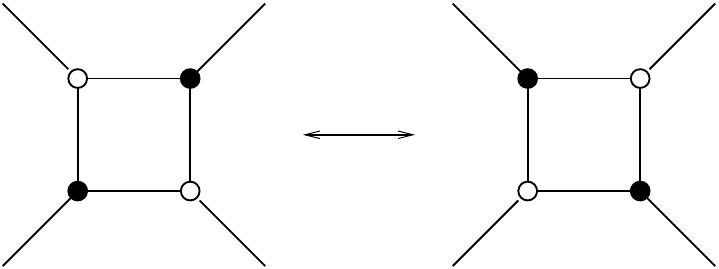}
\hspace{.3in}
\raisebox{6pt}{\includegraphics[height=.3in]{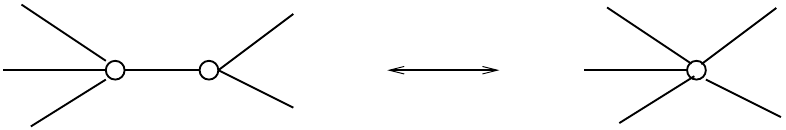}}
\hspace{.3in}
\raisebox{16pt}{\includegraphics[height=.07in]{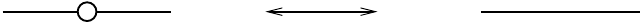}}
\caption{%
	Local moves (M1), (M2), (M3) on plabic graphs.}
\label{fig:M1}
\end{figure}
\begin{defi}\label{def:move}
Two plabic graphs are called \emph{move-equivalent} if they can be obtained
from each other by moves (M1)-(M3).  The \emph{move-equivalence class}
of a given plabic graph $G$ is the set of all plabic graphs which are move-equivalent
to $G$.
A plabic graph 
is called \emph{reduced} if there is no graph in its move-equivalence
in which two adjacent vertices $u$ and $v$ 
are connected by more than one edge
\end{defi}

In particular, given a plabic graph $G$, we can always apply moves to $G$ to obtain a new graph $G'$ which is bipartite.

For a plabic graph $G$, a \emph{trip} is a directed path $T$ in $G$ such that
\begin{itemize}
    \item $T$ either joins two boundary vertices (\emph{one-way trip}) or it is a closed cycle that contains none of the boundary vertices (\emph{round-trip}).
    \item $T$ obeys the following ``rules of the road" at each internal vertex: turn right at a white vertex, and turn left at a black vertex.
\end{itemize}

We associate a decorated permutation $\pi_G$ to each reduced plabic graph $G$, called \emph{trip permutation}, such that $\pi_G(i)$ is the ending vertex of the trip that starts at the boundary vertex $i$.
The moves (M1)-(M3) do not change the trip permutation of a plabic graph, and equivalence classes of reduced plabic graphs are in bijection with positroids \cite{postnikov2006}.

\begin{defi}[Plabic tilings] \label{def:plabictilings}
Let $G$ be any connected reduced plabic graph with $n$ boundary vertices, and let $\mathbf{P}_n$ be a convex $n$-gon, whose vertices are labelled from $1$ to $n$ in clockwise order. The \emph{plabic tiling} $\T(G)$ dual to $G$ is a tiling of $\mathbf{P}_n$ by coloured polygons (bigons allowed) such that: i) it is the planar dual of $G$; ii) each black (white) vertex of $G$ is dual to a black (white) polygon in $\mathcal{T}(G)$; iii) vertex $i$ of $\mathbf{P}_n$ is dual to the face of $G$ touching boundary vertices $i-1$ and $i$. We consider two plabic tilings $\T(G)$ and $\T'(G')$ \emph{equivalent} if $G$ and $G'$ are move-equivalent. 

Conversely, if $\T$ is a plabic tiling, the dual plabic graph 
$G(\T)$ is obtained from $\T$ by placing a black vertex in each black polygon,
a white vertex in each white polygon, and connecting two vertices whenever they correspond
to two polygons which share an edge.
\end{defi}

\printbibliography
\end{document}